\documentclass[a4paper,12pt]{amsart}
\usepackage{amsthm, amsmath, amssymb, mathtools}
\usepackage[top=30truemm, bottom=30truemm, left=25truemm, right=25truemm]{geometry}
\usepackage{url}
\usepackage[dvipdfmx,setpagesize=false]{hyperref}

\allowdisplaybreaks[4]

\numberwithin{equation}{section}
\newcounter{count}

\newcommand{\num}{\stepcounter{count}\the\value{count}}

\newcommand{\N}[0]{\mathbb{N}}
\newcommand{\Z}[0]{\mathbb{Z}}
\newcommand{\R}[0]{\mathbb{R}}

\newtheorem{theorem}{Theorem}[section]
\newtheorem{lemma}[theorem]{Lemma}
\newtheorem{corollary}[theorem]{Corollary}
\newtheorem{proposition}[theorem]{Proposition}

\theoremstyle{definition}
\newtheorem{remark}[theorem]{Remark}
\newtheorem{notation}[theorem]{Notation}
\newtheorem*{acknowledgment}{Acknowledgment}

\makeatletter
\def\@seccntformat#1{\csname the#1\endcsname.\quad}
\makeatother

\title[Simple normality and the Riemann zeta function]{The simple normality of the fractional powers\\ of two and the Riemann zeta function}


\author[Y. Kanado]{Yuya Kanado}
\address{Yuya Kanado\\
	Graduate School of Mathematics\\ Nagoya University\\ Furo-cho\\ Chikusa-ku\\ Nagoya\\ 464-8602\\ Japan}
\email{m21017a@math.nagoya-u.ac.jp}

\author[K. Saito]{Kota Saito}
\address{Kota Saito\\Faculty of Pure and Applied Sciences\\ University of Tsukuba\\ 1-1-1 Tennodai\\ Tsukuba\\ Ibaraki\\ 305-8577\\ Japan}
\email{saito.kota.gn@u.tsukuba.ac.jp}

\begin{document}

\maketitle

\begin{abstract}
A real number is simply normal to base $b$ if its base-$b$ expansion has each digit appearing with average frequency tending to $1/b$. In this article, we discover a relation between the frequency at which the digit $1$ appears in the binary expansion of $2^{p/q}$ and a mean value of the Riemann zeta function on arithmetic progressions. Consequently, we show that 
\[
\lim_{l\to \infty} \frac{1}{l}\sum_{0<|n|\leq 2^l  } \zeta\left(\frac{2 n\pi i}{\log 2}\right) \frac{e^{2n\pi i p/q} }{n} =0
\]
if and only if $2^{p/q}$ is simply normal to base $2$. 
\end{abstract}

\section{Introduction}
Let $\lfloor x\rfloor$ denote the integer part of $x\in \mathbb{R}$. Fix any integer $b\geq 2$. For all $x\in \mathbb{R}$, $a\in \{0,1,\ldots ,b-1\}$, and real numbers $l>0$, we define 
\[
A_b(l;a,x)= \# \{d\in \Z \colon 0\leq d\leq l, \   \lfloor b^d x\rfloor \in a+b\mathbb{Z} \}.  
\]
If $x= \sum_{d=-m}^\infty c_d b^{-d}$ is the $b$-adic expansion of a given real number $x$, then $A_b(l;a,x)$ is equal to the number of $d\in [0,l]$ such that $c_d=a$. We say that $x$ is \textit{simply normal} to base $b$ if for each $a\in \{0,1,\ldots, b-1\}$, we have
\[
\lim_{l\to \infty} A_b (l;a,x)/l = 1/b. 
\]

Borel showed that almost all real numbers are simply normal\footnote{Precisely, he showed that almost all real numbers are normal to base $b$ for every integer $b\geq 2$. Thus, he obtained a much stronger result than the one we exhibit.} to base $b$ for all $b\geq2$ in 1909 \cite{Borel}; however, the simple normality for many non-artificial numbers such as $\pi, e, \log2$, and $\sqrt{2}$ is unknown. In this article, we do not determine whether $2^{p/q}$ is simply normal, but we discover a relation between $A_2(l;1,2^{p/q})$ and a mean value of the Riemann zeta function on vertical arithmetic progressions. Let $\zeta(s)$ denote the Riemann zeta function.  

\begin{theorem}\label{theorem:main1}
Let $p$ and $q$ be relatively prime integers with $1\leq p<q$. Then we have 
\[
A_2(l;1, 2^{p/q}) = \frac{l}{2} - \frac{1}{2\pi i } \sum_{0<|n|\leq 2^l } \zeta\left(\frac{2 n\pi i}{\log 2}\right) \frac{e^{2n\pi i  p/q} }{n} +o(l)\quad (\text{as $l\to\infty$}),
\]
where $l$ runs over positive real numbers. Especially, we have
\begin{equation}\label{equation:important-lim}
    \lim_{l\to \infty} \frac{1}{l}\sum_{0<|n|\leq 2^l  } \zeta\left(\frac{2 n\pi i}{\log 2}\right) \frac{e^{2n\pi i p/q} }{n} =0
\end{equation}
if and only if $2^{p/q}$ is simply normal to base $2$. 
\end{theorem}

It is unknown whether $A_2(l;1, 2^{p/q})/l$ converges as $l$ tends to infinity. Theorem~\ref{theorem:main1} also reveals that the limit on the left-hand side of \eqref{equation:important-lim} exists if and only if $A_2(l;1, 2^{p/q})/l$ converges. Moreover, if we have
\begin{equation}\label{inequality:zeta-beta}
    \limsup_{l\to\infty} \frac{1}{l} \left\lvert \frac{1}{2\pi i} \sum_{0<|n|\leq 2^l  } \zeta\left(\frac{2 n\pi i}{\log 2}\right) \frac{e^{2n\pi i p/q} }{n} \right\rvert<\beta
\end{equation}
for some real number $\beta\in (0,1/2]$, then $ 1/2-\beta <  A_2 (l;1,2^{p/q})/l < 1/2+\beta$ holds for sufficiently large $l>0$. 

\begin{remark}
Since $\zeta(\overline{s})=\overline{\zeta(s)}$, we obtain 
\begin{equation}\label{equation:realsum}
    \frac{1}{2\pi i}\sum_{0<|n|\leq 2^l  } \zeta\left(\frac{2 n\pi i}{\log 2}\right) \frac{e^{2n\pi i p/q} }{n} =\frac{1}{\pi} \sum_{0<n\leq 2^l  } \Im\left( \zeta\left(\frac{2 n\pi i}{\log 2}\right) \frac{e^{2n\pi i p/q} }{n}\right),
\end{equation}
where $\Im (z)$ denotes the imaginary part of $z$ for all $z\in \mathbb{C}$; thus, the left-hand side of \eqref{equation:realsum} is always a real number. 
\end{remark}

It is natural to investigate a mean value of the Riemann zeta function on arithmetic progressions to verify \eqref{equation:important-lim} or \eqref{inequality:zeta-beta}. When $0<\Re(s_0)<1$, there is research on asymptotic formulas of $\sum_{0\leq n<M} \zeta(s_0 + idn)$. For example, Steuding and Wegert firstly studied the asymptotic formulas for all $d=2\pi/\log k$ with $k\in \mathbb{Z}_{\ge 2} $ \cite[Theorem~1.1]{Steuding-Wegert}. Furthermore, in \cite{Ozbek-Steuding_2017,Ozbek-Steuding_2019}, \"{O}zbek and Steuding showed  that for all $s_0\in \mathbb{C}$ with $\Re(s_0)\in (0,1)$ 
\begin{equation}\label{equation-OS}
\lim_{M\to\infty} \frac{1}{M}\sum_{0\leq n<M}\zeta\left(s_0+ind\right)=
\begin{cases}
    (1-k^{-s_0})^{-1} & \text{ if }d=\frac{2\pi r}{\log k},\ r\in \mathbb{N},\ k\in \mathbb{Z}_{\ge 2}, \\
    1 & \text{ otherwise},
\end{cases}
\end{equation}
Here the form $d= 2\pi r/ (\log k)$ with $k\in \mathbb{Z}_{\ge 2}$ and $r\in \mathbb{N}$, then $r$ is supposed to be the smallest integer for which such a value $k$ exists. They also gave similar asymptotic formulas on more general arithmetic progressions \cite{Ozbek-Steuding_2019}. We get the following: none obtained asymptotic formulas on $\Re(s_0)=0$. 

\begin{theorem}\label{theorem:main2}
Let $k$ be an integer not less than $2$. For every real number $l\geq 2$, we have 
\begin{equation}\label{equation:main2}
    \frac{1}{2\pi i }\sum_{0<|n|\leq k^l} \zeta \left(\frac{2n\pi i}{\log k}\right) \frac{1}{n}=O_k(1). 
\end{equation}
\end{theorem}

The summations in \eqref{equation:important-lim} and \eqref{equation:main2} are slightly different from \eqref{equation-OS}, and hence we have to pay attention when comparing them. In Remark~\ref{remark:comparing}, essentially by \eqref{equation-OS}, for all $p,q\in \mathbb{Z}$, $k\in \mathbb{Z}_{\ge 2}$, and $\sigma_0\in (0,1)$, we will see that   
\begin{equation}\label{equation-OS-modifying}
        \lim_{l\to \infty} \frac{1}{l}\sum_{0<|n|\leq k^l } \zeta\left(\sigma_0+\frac{2n\pi i}{\log k}\right) \frac{e^{2n\pi i  p/q}}{n}=0.
\end{equation}
 Therefore, from Theorem~\ref{theorem:main1}, transferring \eqref{equation-OS-modifying} with $k=2$ to the case $\sigma_0=0$ is equivalent to verifying the simple normality of $2^{p/q}$. Moreover, we can consider Theorem~\ref{theorem:main2} a successful transfer \eqref{equation-OS-modifying} with $p=q=1$ to $\sigma_0=0$.

We also find research on the high moments of the Riemann zeta function. Good showed asymptotic formulas for the fourth moment on vertical arithmetic progressions belonging to the right half of the critical strip \cite{Good}. Kobayashi presented the ones for the second moments of $\zeta(1/2+in)$ \cite{Kobayashi}. We do not study relations between problems on digits and the high moments of the Riemann zeta function. In the future, it would be interesting if we discovered their connections. Further, we only focus on the Riemann zeta function in the article. It would be attractive if we disclosed connections between problems on digits and other zeta functions such as the $L$-function, Hurwitz zeta function, Dedekind zeta function, multiple zeta function, etc.  

\begin{notation}
Let $\mathbb{N}=\{1,2,3,\ldots\}$. For every $m\in \mathbb{Z}$, we define $\mathbb{Z}_{\ge m}$ as the set of integers not less than $m$. For $x\in \mathbb{R}$, let $\{x\}$ denote the fractional part of $x$, and $\|x\|$ denote the distance from $x$ to the nearest integer. Let $\log_k x$ be $\log x/\log k$ for every $x>0$ and integer $k\geq2$. We say that $f(x)=g(x)+o(h(x))$ (as $x\to \infty$) if for all $\epsilon>0$ there exists $x_0>0$ such that $|f(x)-g(x)|\leq h(x)\epsilon $ for all $x\geq x_0$. If $x_0$ depends on some parameters $\epsilon, a_1,\ldots ,a_n$, then we write $f(x)=g(x)+o_{a_1,\ldots , a_n}(h(x))$. We also say that $f(x)=g(x)+O(h(x))$ for all $x\geq x_0$ if there exists $C>0$ such that $|f(x)-g(x)|\leq Ch(x)$  for all $x\geq x_0$. If $C$ depends on some parameters $a_1,\ldots , a_n$, then we write  $f(x)=g(x)+O_{a_1,\ldots ,a_n}(h(x))$ for all $x\geq x_0$.

\end{notation}
\section{A Preliminary discussion}
In this section, we will observe that the following theorem implies Theorem~\ref{theorem:main1}. In addition, we will introduce a certain arithmetic function which plays a key role in the proof. 

\begin{theorem}\label{theorem:generalization}
Let $p$ and $q$ be relatively prime positive integers with $1\leq p< q$. Let $k\geq 2$ be an integer which is not a $q$-th power of integer. Then we have
\begin{equation}\label{equation:main-formula}
    \sum_{0\leq d\leq l} \{ k^{d+p/q} \} = \frac{l}{2} -\frac{1}{2\pi i }\sum_{0<|n|\leq k^l} \zeta\left(\frac{2n\pi i}{\log k}\right) \frac{e^{2n\pi i p/q}}{n}+o_{p,q,k}(l) \quad (\text{as }l\to \infty ),
\end{equation}
where $l$ runs over positive real numbers. 
\end{theorem}
We aim to give a proof of Theorem~\ref{theorem:generalization}. Roughly speaking, by substituting $p=q=1$ in Theorem~\ref{theorem:generalization}, the first term $l/2$ on the right-hand side of \eqref{equation:main-formula} vanishes and we obtain Theorem~\ref{theorem:main2}. 
In Section~\ref{section:evaluation-S}, we will prove Theorem~\ref{theorem:main2} by verifying the substitution. In Section~\ref{section:completion}, we will prove Theorem~\ref{theorem:generalization}.

\begin{remark} By the definition of $o_{p,q,k}(1)$, for all $\epsilon>0$ there exists $l_0=l_0(\epsilon, p,q,k)>0$ such that for all $l\geq l_0$ we have
\[
\frac{1}{l} \left\lvert \sum_{0\leq d\leq l} \{ k^{d+p/q} \} - \frac{l}{2} +\frac{1}{2\pi i }\sum_{0<|n|\leq k^l} \zeta\left(\frac{2n\pi i}{\log k}\right) \frac{e^{2n\pi i p/q}}{n} \right\rvert \leq \epsilon.
\]
The constant $l_0$ is non-computable since we will apply Ridout's theorem (Theorem~\ref{theorem:ridout}), which involves Diophantine approximations. The finiteness of Ridout's theorem is proven by an ineffective method. 
\end{remark}

\begin{remark}\label{remark:comparing}
To compare our results with \eqref{equation-OS}, let us give a proof of \eqref{equation-OS-modifying}. Fix any $s_0\in \mathbb{C}$ with $0<\Re(s_0)<1$. We define $C_d(s_0)$ as the right-hand side of \eqref{equation-OS}. Let $q\in \mathbb{N}$ and $a\in \{0,1,\ldots ,q-1\}$. Then, by \eqref{equation-OS}, partial summation, and $\zeta(\overline{s})=\overline{\zeta(s)}$, for each sufficiently large $M\in \mathbb{N}$, we have
\[
\frac{1}{2\pi i}\sum_{\substack{0<|n|\leq M \\ n \equiv a \text{ mod } q }  } \frac{\zeta(s_0+ind)}{n} = \frac{\Im(C_{qd}(s_0+ida))}{q\pi } \log M +o_{s_0,q,a,d}(\log M). 
\]
If $d=2\pi / \log k$ for some $k\in \mathbb{Z}_{\ge 2}$ and $s_0=\sigma_0\in (0,1)$, then we obtain 
\[
\Im(C_{qd}(s_0+ida))= \Im((1-k^{\sigma_0}e^{2\pi i a} )^{-1})=0.
\]
Thus, we have 
\begin{align*}
    &\frac{1}{2\pi i}\sum_{0<|n|\leq k^l } \zeta\left(\sigma_0+\frac{2n\pi i}{\log k}\right) \frac{e^{2n\pi i  p/q}}{n} \\
    &= \frac{1}{2\pi i} \sum_{a=0}^{q-1} e^{2a\pi i p/q} \sum_{\substack {0<|n|\leq k^l \\ n\equiv a \text{ mod }q }} \zeta\left(\sigma_0+\frac{2n\pi i}{\log k}\right) \frac{1}{n}= o_{\sigma_0,q,k}(l).
\end{align*}
Therefore, we conclude \eqref{equation-OS-modifying}.
\end{remark}

We state $f(X) \ll g(X)$ and $f(X) \ll_{a_1,\ldots, a_n} g(X)$  as $f(X)=O(g(X))$ and $f(X)=O_{a_1,\ldots, a_n}(g(X))$ respectively, where $g(X)$ is non-negative. In addition, we state $f(X)\asymp g(X)$ if $f(X)\ll g(X)\ll f(X)$. 
 
Let us fix $p$ and $q$ as relatively prime integers with $1\leq p\leq  q$. Let $k\geq 2$ be an integer which is not a $q$-th power of integers if $q\geq 2$. We consider the parameters $p$, $q$, and $k$ as constants. Thus, we omit the dependencies of these parameters.

\begin{lemma}\label{lemma:binary-expansion} For all $l\in \mathbb{N}$, we have
\[
\sum_{0\leq m\leq l}\{2^{m+p/q}\}=\sum_{\substack{0\leq d\leq l\\ \lfloor2^{d+p/q}\rfloor\in1+2\Z_{\geq0}}}1+O(1).
\]
\end{lemma}
\begin{proof} Let $\sum_{d=0}^\infty c_d2^{-d}$ be the binary expansion of $2^{p/q}$. Then, for all $m\geq0$, we have
\[
\{2^{m+p/q}\}= \{2^{m} 2^{p/q}\}=\left\{\sum_{d=0}^\infty c_d2^{m-d}\right\}=\sum_{d=m+1}^\infty c_d2^{m-d}=\sum_{d=1}^\infty c_{m+d}2^{-d},
\]
and hence
\begin{align*}
    \sum_{0\leq m\leq l}\{2^{m+p/q}\}
    &=\sum_{0\leq m\leq l}\sum_{d=1}^\infty c_{m+d}2^{-d}
    =\sum_{1\leq k\leq l}c_k\sum_{j=1}^k2^{-j}+\sum_{l+1\leq k}c_k\sum_{j=k-l}^k2^{-j}\\
    &=\sum_{1\leq k\leq l}c_k(1-2^{-k})+\sum_{l+1\leq k}c_k2^{-k+l+1}(1-2^{-l-1})\\
    &=\sum_{0\leq k\leq l}c_k+O(1)
    =\sum_{\substack{0\leq k\leq l;\\ \lfloor2^{k+p/q}\rfloor\in1+2\Z_{\geq0}}}1+O(1).
\end{align*}
\end{proof}

\begin{proof}[Proof of Theorem~\ref{theorem:main1} assuming Theorem~\ref{theorem:generalization}] Fix arbitrary integers $1\leq p<q$ with $\gcd(p,q)=1$. By combining Theorem~\ref{theorem:generalization} with $k=2$ and Lemma~\ref{lemma:binary-expansion}, we obtain Theorem~\ref{theorem:main1}.   
\end{proof}

For every $l\in \N$, we define
\[
A(l) =\sum_{\substack{0\leq d\leq l}}  \{k^{d+p/q} \}. 
\]
The goal of proving Theorem~\ref{theorem:generalization} is to obtain an asymptotic formula of $A(l)$.
For all $\alpha>1$ and $\Re(s)>0$, we define $\varphi(\alpha,s)=\sum_{n=0}^\infty \alpha^{-ns}$.  We set
\begin{align*}
    b(n)=b_k(n)=
    \begin{cases}
	1-k  &  \textup{if } k\mid n  ,\\
	1  &  \textup{otherwise}.
    \end{cases}
\end{align*}
Furthermore, for all $\Re(s)>1$, we define
\[
\eta(s)=\eta_k(s):=\sum_{n=1}^\infty\frac{b_k(n)}{n^s}=(1-k^{1-s})\zeta(s).
\]
Remark that $\eta_k(s)$ is coincident with the eta function $(1-2^{1-s})\zeta(s)$ if $k=2$. Then for every $\Re(s)>1/q$, it follows that 
\begin{align*}
    \varphi(k,qs)\eta(qs)
    &=\left(\sum_{n=0}^\infty\frac{1}{k^{qns}}\right)\left(\sum_{n=1}^\infty\frac{b(n)}{n^{qs}}\right)
    =\left(\sum_{n=1}^\infty\frac{f(n)}{n^s}\right)\left(\sum_{n=1}^\infty\frac{g(n)}{n^s}\right),
\end{align*}
where 
\begin{align*}
    f(d):=
    \begin{cases}
        1  &  \textup{if } \exists n\in \Z_{\ge 0} \textup{ s.t. } d=k^{qn},\\
	0  &  \textup{otherwise},
    \end{cases}\quad 
    g(d):=
    \begin{cases}
        b(n)  &  \textup{if } \exists n\in \Z_{>0} \textup{ s.t. } d=n^q,\\
	0  &  \textup{otherwise}.
    \end{cases}
\end{align*}
For every $n\in \mathbb{N}$, we define $h(n)=\sum_{d\mid n } f(d)g(n/d)$. Then the Dirichlet multiplication leads to
\begin{equation}\label{equation:Dirichlet-series}
\varphi(k,qs)\eta(qs)=\sum_{n=1}^\infty \frac{h(n)}{n^s}. 
\end{equation}

\begin{lemma}\label{Lemma:Sum-of-h} For every $x\geq 2$, we have
\begin{equation}\label{equation:key-sum}
    \sum_{n\leq x}h(n)=(k-1)\sum_{\substack{0\leq d\leq q^{-1}\log_k x}}\{x^{1/q}/k^d\}+ O(1).
\end{equation}
\end{lemma}
\begin{proof}
By the definition of $f(\cdot)$ and $g(\cdot)$, it follows that 
\[
h(n) =\sum_{d\mid n} f(d)g(n/d) = \sum_{\substack{d\geq 0 \\ k^{qd}\mid n}} g(n/k^{qd}),
\]
and hence 
\[
\sum_{n\leq x}h(n)=\sum_{n\leq x}\sum_{\substack{d\geq0 \\ k^{qd}\mid n}}g(n/k^{qd})=\sum_{0\leq d\leq q^{-1}\log_k x}\sum_{1\leq n\leq x/k^{qd}}g(n).
\]
In addition, the definitions of $g(\cdot)$ and $b(\cdot)$ yield 
\begin{align*}
    \sum_{1\leq n\leq x/k^{qd}}g(n)
    &=\sum_{1\leq j^q\leq x/k^{qd}}b(j)=\sum_{1\leq j\leq x^{1/q} /k^{d}} b(j)= \lfloor x^{1/q}/k^d \rfloor- k\lfloor x^{1/q}/k^{d+1} \rfloor  \\
    &= -\{x^{1/q}/k^d\}+k\{x^{1/q}/k^{d+1} \}. 
\end{align*}
Therefore, we conclude \eqref{equation:key-sum}.
\end{proof}

By applying Lemma~\ref{Lemma:Sum-of-h} with $x=k^{ql+p}$, we observe that
\begin{align} \label{equation:numberof1}
    \sum_{n\leq x} h(n) &= (k-1)\sum_{0\leq d\leq l }\{ k^{(l-d)+p/q} \} +O(1)\\
    &=(k-1)\sum_{0\leq d\leq l } \{k^{d+p/q} \} +O(1)=(k-1)A(l)+O(1), \nonumber
\end{align}
and hence, the mean value of $h(n)$ is directly connected to  $A(l)$.

\section{Outline of the proof of Theorem~\ref{theorem:generalization}}\label{section:sketch}

For simplicity, we do not consider the case $q=1$ in this section. Thus, the integers $p$ and $q$ are relatively prime with $1\leq p<q$, and $k$ is an integer larger than or equal to $2$ which is not a $q$-th power of integers. Let $l\in \N$ be a sufficiently large parameter, and let $x=k^{ql+p}$. We will first apply Perron's formula to obtain an asymptotic formula of $\sum h(n)$. 

\begin{lemma}[Perron's formula]\label{Lemma:Perrons-formula}
Let $\alpha(s)$ be the Dirichlet series of the form $\alpha(s)=\sum_{n=1}^\infty a_n n^{-s}$. Let $\sigma_a$ be the abscissa of absolute convergence of $\alpha(s)$. If $c>\max (0,\sigma_a)$, $x>0$, and $T> 0$, then we have
\[
\sum_{n\leq x}{}^\prime a_n =\frac{1}{2\pi i} \int_{c-iT}^{c+iT} \alpha(s) \frac{x^s}{s} ds +R   
\]
and 
\[
R\ll \sum_{\substack{x/2<n<2x\\ n\not=x}}|a_n|\min\left(\frac{x}{T|x-n|},1\right) +\frac{4^{c}+x^{c}}{T}\sum_{n=1}^\infty\frac{|a_n|}{n^c},
\]
where $\sum_{n\leq x}^{\prime}$ indicates that if $x$ is an integer, then the last term is to be counted with weight $1/2$. 
\end{lemma}
\begin{proof}
See \cite[Theorem~5.2, Corollary~5.3]{Montgomery-Vaughan} .
\end{proof}

Recall that the corresponding Dirichlet series of $h(n)$ is $\varphi(k, qs) \eta (qs)$ from \eqref{equation:Dirichlet-series}. Therefore, for $c>1/q$ and $T>0$, Lemma~\ref{Lemma:Perrons-formula} with $a_n=h(n)$ implies 
\begin{equation}\label{asymp:h-to-Integral}
    \sum_{n\leq x}\: h(n) = \frac{1}{2\pi i} \int_{c-iT}^{c+iT} \varphi(k,qs) \eta(qs) \frac{x^s}{s} ds+(\text{errors}).
\end{equation}
The summation $ \sum_{n\leq x}$ should be written as $ \sum_{n\leq x}'$, but we ignore the gaps between these sums. Let us also skip to evaluate all the errors. In Section~\ref{section:perron}, we will do a precise discussion on \eqref{asymp:h-to-Integral}. By the definitions of $\varphi$ and $\eta$, 
\[
\varphi(k,qs) \eta(qs)\frac{x^s}{s}= \frac{1-k^{1-qs}}{1-k^{-qs}} \zeta (qs)\frac{x^s}{s}(\eqqcolon \Phi(s; x)).
\]
Let $c=c(l)>1/q$ and $T=T(l)>0$ be suitable parameters. By substituting $x=k^{ql+p}$,  the equations \eqref{equation:numberof1} and \eqref{asymp:h-to-Integral} yield that
\begin{equation}\label{asymp:01}
    (k-1) A(l) =  \sum_{n\leq k^{ql+p} } h(n) +O(1)=\frac{1}{2\pi i} \int_{c-iT}^{c+iT} \Phi (s; k^{ql+p}) ds+(\text{errors}).
\end{equation}
We shall apply the residues theorem similarly to the analytic proof of the prime number theorem (see \cite[Chapter~6]{Montgomery-Vaughan}). We move the vertical integral from $\int_{c-iT}^{c+iT}$ to $\int_{\sigma-iT}^{\sigma+iT}$ for some fixed $\sigma<0$, where we will take $\sigma=-1/(2q)$ in Section~\ref{section:evaluation-S}. The residues of $\Phi(s;k^{ql+p})$ are 
\begin{equation}\label{residues}
    \left\{
    \begin{aligned}
        \displaystyle{(1-k)\zeta\left(\frac{2n\pi i}{\log k}\right) \frac{e^{2n\pi i p/q}}{2n\pi i}} &\quad \text{at}\quad s=\frac{2n\pi i}{q\log k} \text{ for } n\neq 0,\\
        (k-1)\frac{l}{2} +O(1) &\quad \text{at}\quad s=0.
    \end{aligned}
    \right.
\end{equation}
We will observe \eqref{residues} in Section~\ref{section:perron} and Section~\ref{section:evaluation-RandS}. Therefore, by applying the residue theorem,
\begin{align} \nonumber
    A(l)&= \frac{1}{2\pi i(k-1)} \int_{c-iT}^{c+iT} \Phi(s; k^{ql+p}) ds+(\text{errors})\\ \label{equation:sketch1}
    &= \frac{l}{2} -\frac{1}{2\pi i }\sum_{0<|n|\leq \frac{\log k}{2\pi }T} \zeta\left(\frac{2n\pi i}{\log k}\right) \frac{e^{2n\pi i p/q}}{n}\\ \nonumber
    &\hspace{50pt}+\frac{1}{2\pi i(k-1)} \int_{\sigma-iT}^{\sigma+iT} \Phi(s;k^{ql+p})  ds+(\text{errors}) .
\end{align}
We will calculate the errors in Section~\ref{section:evaluation-RandS}. 

By the functional equation of the Riemann zeta function, we have  $\zeta (s) =\chi (s) \zeta(1-s)$, where $\chi(s)=2^{s-1}\pi^s\sec(\pi s/2)/\Gamma(s)$. Therefore, recalling the definition of $\Phi$, we see that
\[
\frac{1}{2\pi i(k-1)}
\int_{\sigma-iT}^{\sigma+iT} \Phi(s;k^{ql+p})  ds = \frac{1}{2\pi i(k-1)}
\int_{\sigma-iT}^{\sigma+iT} \frac{1-k^{1-qs}}{1-k^{-qs}} \chi(qs) \zeta(1-qs) k^{s(ql+p)} \frac{ds}{s}.
\]
In addition, for every $\Re(s)<0$, 
we observe that 
\begin{align}\label{equation:geoseries}
    \frac{1-k^{1-qs}}{1-k^{-qs}}&=(k^{1-qs}-1)\cdot \frac{k^{qs}}{1-k^{qs}}=(k^{1-qs}-1) \left(\sum_{m=1}^\infty k^{qms} \right)\\ \nonumber
    &= \sum_{m=1}^\infty (k^{1-qs}-1)\cdot k^{qms}
    = \sum_{m=1}^\infty k\cdot k^{(m-1)qs} - \sum_{m=1}^\infty k^{qms} = \sum_{m=0}^\infty a_m k^{qms},
\end{align}
where $a_0=k$, and $a_m=k-1$ for every $m\geq 1$. Therefore,  by choosing $\sigma=-1/(2q)<0$, 
\begin{align*}
    &\frac{1}{2\pi i(k-1)} \int_{\sigma-iT}^{\sigma+iT} \Phi(s;k^{ql+p})  ds\\
    &=\sum_{m=0}^\infty \sum_{n=1}^\infty \frac{a_m}{2\pi i (k-1)} \int_{\sigma-iT}^{\sigma+iT}  k^{qms}  n^{qs-1}k^{s(ql+p)} \chi(qs)\frac{ds}{s}\\
    &=\sum_{m=0}^\infty \sum_{n=1}^\infty \frac{a_m}{2\pi  (k-1)} k^{-(m+l+p/q)/2} n^{-3/2} \int_{-T}^{T}  (k^{m+l+p/q} n )^{iqt}  \frac{\chi(-1/2+iqt)}{-1/(2q)+it } dt.
\end{align*}
In Section~\ref{section:evaluation-S}, we will apply the following lemmas to calculate exponential integrals to find an asymptotic formula of the above integral.
\begin{lemma}[the first derivative test]\label{lemma:exponential-sums}
Let $F(x)$ be a real differentiable function defined on $[a,b]$ such that $F'(x)$ is monotonic throughout the interval $[a,b]$. Suppose that there exists $M>0$ such that for every $x\in [a,b]$, we have $|F(x)|\geq M$. Then 
\[
\left\lvert \int_a^be^{iF(x)}dx\right\rvert \leq\frac{4}{M}.
\]
\end{lemma}
\begin{proof}
See {\cite[Lemma~4.2]{Titchmarsh}}.  
\end{proof}
\begin{lemma}[the second derivative test]\label{lemma:second-derivative}
Let $F(x)$ be a twice differentiable real function defined on $[a,b]$. Suppose that there exists $r>0$ such that for every $x\in [a,b]$, we have $|F''(x)|\geq r$. Then 
\[
\left\lvert \int_a^be^{iF(x)}dx\right\rvert \leq \frac{8}{r^{1/2}}.
\]
\end{lemma}
\begin{proof}
See {\cite[Lemma~4.4]{Titchmarsh}}.  
\end{proof}
\begin{lemma}[the stationary phase method]\label{lemma:stationary-phase}
Let $F(x)$ be a real-valued function defined on $[a,b]$ which is differentiable up to the third order. Suppose that there exist $\lambda_2, \lambda_3>0$ and $A>0$ such that for every $x\in [a,b]$, we have 
\begin{gather} \label{condition:stationary1}
0<\lambda_2 \leq -F''(x) < A \lambda_2  \\ \label{condition:stationary2} 
|F'''(x)| <A\lambda_3. 
\end{gather}
Let $F'(c)=0$, where $c\in [a,b]$. Then 
\begin{align*}
    \int_a^b e^{iF(x)} dx &= 
    (2\pi)^{\frac{1}{2}} \frac{e^{-\pi i/4 +iF(c)} }{|F''(c)|^{1/2}}  +O(\lambda_2^{-\frac{4}{5}}\lambda_3^{\frac{1}{5}} )\\ &\quad +O\left(\min\left(|F'(a)|^{-1},\lambda_2^{-\frac{1}{2}} \right)\right) +O\left(\min\left(|F'(b)|^{-1},\lambda_2^{-\frac{1}{2}} \right)\right) . 
\end{align*}
\end{lemma}
\begin{proof}
See {\cite[Lemma~4.6]{Titchmarsh}}. 
\end{proof}
By applying Lemmas~\ref{lemma:exponential-sums} to \ref{lemma:stationary-phase} in Section~\ref{section:evaluation-S}, we will show that 
\begin{equation}\label{equation:sketch2}
    \frac{1}{2\pi i(k-1)} \int_{\sigma-iT}^{\sigma+iT} \Phi(s;k^{ql+p})  ds=\frac{(q-1)l}{2} -\sum_{0\leq m\leq (q-1)l} \{k^{m+l+p/q} \}+(\text{errors}).
\end{equation}
Here the errors on the right-hand side contain 
\begin{equation}\label{errors-of-sawtooth}
    \sum_{0\leq m\leq (q-1)l} \min \left( \frac{1}{2}, \frac{C}{lk^{(q-1)l-m}\|k^{m+l}\cdot k^{p/q} \|}\right)
\end{equation}
for some constant $C>0$. The error \eqref{errors-of-sawtooth} comes from the partial Fourier sums of the sawtooth function. To investigate lower bounds for $\|k^{m+l}\cdot k^{p/q} \|$, we will apply Ridout's theorem in Section~\ref{section:completion}. Let $\epsilon$ be an arbitrarily small positive real number. By the theorem, for every $a\in \mathbb{Z}$ and $0\leq m\leq (q-1)l$, we have   
\begin{equation}\label{inequality:diophantine}
    \left\lvert k^{p/q}- \frac{a}{k^{m+l}} \right\rvert \gg_{\epsilon,k,p,q} k^{-(1+\epsilon) (m+l)},
\end{equation}
where the implicit constant is ineffective. By applying \eqref{inequality:diophantine}, we will show that \eqref{errors-of-sawtooth} is small enough. Therefore, combining \eqref{equation:sketch1} and  \eqref{equation:sketch2} presents
\begin{align*}
    &\sum_{0\leq m\leq l} \{k^{m+p/q}\}= \frac{ql}{2} -\frac{1}{2\pi i }\sum_{0<|n|\leq \frac{\log k}{2\pi }T} \zeta\left(\frac{2n\pi i}{\log k}\right) \frac{e^{2n\pi i p/q}}{n} -\sum_{0\leq m\leq (q-1)l} \{k^{m+l+p/q} \}+(\text{errors}),
\end{align*}
which completes 
\[
A(ql) = \frac{ql}{2}-\frac{1}{2\pi i }\sum_{0<|n|\leq \frac{\log k}{2\pi }T} \zeta\left(\frac{2n\pi i}{\log k}\right) \frac{e^{2n\pi i p/q}}{n} +(\text{errors}).
\]
Interestingly, we discover a relation between  
\[
\sum_{0\leq m\leq l} \{k^{m+p/q}\} \quad \text{and}\quad  -\sum_{0\leq m\leq (q-1)l} \{k^{m+l+p/q} \}
\]
through the functional equation $\zeta(qs)= \chi(qs) \zeta (1-qs)$, one of the key ingredients of the proof. 

We organize the remainder of the article as follows. In Section \ref{section:perron}, we apply Perron's formula and calculate the residues of $\Phi(s)$. In Section~\ref{section:evaluation-RandS}, we move the vertical integral from $\int_{c-iT}^{c+iT}$ to $\int_{\sigma-iT}^{\sigma+iT}$ for some fixed $\sigma<0$ and provide \eqref{equation:sketch1}. Section~\ref{section:evaluation-S} shows \eqref{equation:sketch2} using the functional equation, and Lemmas~\ref{lemma:exponential-sums} to \ref{lemma:stationary-phase}. At last, in Section~\ref{section:completion}, we complete the proof of Theorem~\ref{theorem:generalization}.

\section{Applying Perron's formula and the residue theorem}\label{section:perron}
From this section, we also consider the case $q=1$. Recall that we fix arbitrary integers $k\in \mathbb{Z}_{\ge 2}$, $p$, and $q$ with $\textup{gcd}(p,q)=1$ and $1\leq p\leq q$.  Assume that $k$ is not a $q$-th power of integers if $q\geq 2$. Let $x\geq 2$ and $T\geq 2$ be sufficiently large parameters. We will choose $x=k^{ql+p}$ and $T\asymp lk^{ql}$. Let $c>1/q$ be a parameter that depends on $x$. We will choose $c=1+1/\log x$ later. By Lemma~\ref{Lemma:Perrons-formula} (Perron's formula) and the definition of $h(n)$, we obtain  
\begin{equation}\label{equation:Perron-h}
    \sum_{n\leq x}{}^\prime h(n) = \frac{1}{2\pi i} \int_{c-iT}^{c+iT} \varphi(k,qs) \eta(qs) \frac{x^s}{s} ds+R,
\end{equation}
where
\begin{equation}\label{inequality:R}
    R\ll  \sum_{\substack{x/2<n<2x\\ n\not=x}}|h(n)|\min\left(\frac{x}{T|x-n|},1\right) +\frac{4^{c}+x^{c}}{T}\sum_{n=1}^\infty\frac{|h(n)|}{n^c}.
\end{equation}
To transfer the vertical line of the integral, we should investigate the poles and residues of $\varphi(k, qs) \eta(qs) x^s/s$. Recall that $\Phi(s)=\Phi(s;x)=\varphi(k,qs) \eta(qs) x^s/s$.

\begin{lemma}
Let $s_n= 2n \pi i / (q\log k)$ for every $n\in \Z$. The function $\Phi(s)$ has a pole at $s=s_n$ for every $n\in \Z$. In addition, 
for every $n\in \Z $, the residue of $\Phi(s)$ at $s=s_n$ is  
\begin{equation}
    \left\{
    \begin{aligned}
        \frac{1}{q\log k}\cdot \frac{\eta(qs_n)}{s_n}x^{s_n}  & \quad\textup{if } n \neq 0, \\
        (k-1)\frac{\log x}{2q\log k} +O(1) &\quad \textup{if } n =0.
    \end{aligned}
    \right.
\end{equation}
\end{lemma}
\begin{proof} We observe that
\[
\varphi(k,qs)=\sum_{n=0}^\infty \frac{1}{k^{qs}} =\frac{1}{1-k^{-qs}}.
\]
Thus, the function $\varphi(k^q,s)$ has a simple pole at $s=s_n$ for every $n\in \Z$. For every $0<|s-s_n|<\epsilon$, we have 
\begin{align*}
    \varphi(k, qs)
    &=\frac{1}{1-k^{-q(s-s_n)}}\\
    &=\frac{1}{q(\log k)(s-s_n)}\cdot\frac{1}{1-\frac{q\log k}{2}(s-s_n)+O(|s-s_n|^2)}\\
    &=\frac{1}{q(\log k)(s-s_n)}\left(1+\frac{q\log k}{2}(s-s_n)+O(|s-s_n|^2)\right)\\
    &=\frac{1}{q(\log k)(s-s_n)}+\frac{1}{2}+O(|s-s_n|),
\end{align*}
which implies that the residue of $\varphi(k,qs)$ at $s=s_n$ is equal to $1/(q\log k)$. Therefore, for every $n\in \Z \setminus \{0\}$, the residue of $\Phi(s)$ at $s=s_n$ is equal to 
\begin{align*}
    \frac{1}{q\log k}\cdot \frac{\eta(qs_n)}{s_n}x^{s_n}.
\end{align*}
When $n=0$, for every $0<|s|<\epsilon$, we have
\begin{align*}
    \Phi(s)
    &=\left(\frac{1}{q(\log k)s}+\frac{1}{2}+O(|s|)\right)(\eta(0)+q\eta^\prime(0)s+O(|s|^2))\left(\frac{1}{s}+\log x+O(|s|)\right)\\
    &=\frac{1}{q\log k}\frac{1}{s^2} +\left(\frac{\eta(0)\log x+q\eta^\prime(0)}{q\log k}+\frac{\eta(0)}{2} \right) \frac{1}{s} +O(1).
\end{align*}
Since $\eta(0)=(1-k)\zeta(0)=(k-1)/2$, the residue of $\Phi(s)$ at $s=s_0(=0)$ is 
\[
(k-1)\frac{\log x} {2q\log k}+O(1).
\]
\end{proof}

Let $\sigma_0$ be a negative constant depending only on $q$. We will choose $\sigma_0=-1/(2q)$ later. Let $\delta$ be a sufficiently small absolute constant belonging to $(0,1/2)$, and we define
\[
\mathcal{T}=\mathcal{T}_{\delta}:=\bigsqcup_{n=0}^\infty\left(\frac{2n\pi}{q\log k}-\delta,\frac{2n\pi}{q\log k}+\delta\right).
\] 
To avoid the poles of $\Phi(s)$, if necessary, we assume that $T\in [2,\infty)\setminus \mathcal{T}$. Then, by applying the residue theorem and \eqref{equation:Perron-h}, we obtain 
\begin{equation}\label{equation:PerronResidue}
    \sum_{n\leq x}{}^\prime h(n)=(k-1)\frac{\log x}{2q\log k}+\frac{1}{q\log k}\sum_{0<|s_n|\leq T}\frac{\eta(qs_n)}{s_n}x^{s_n}+R+S_0+S_1+O(1),
\end{equation}
where
\begin{align*}
    S_0\coloneqq \frac{1}{2\pi i}\int_{\sigma_0-iT}^{\sigma_0+iT}\Phi(s)ds,\quad S_1\coloneqq \frac{1}{2\pi i}\left(\int_{\sigma_0+iT}^{c+iT}-\int_{\sigma_0-iT}^{c-iT}\right)\Phi(s) ds.
\end{align*}

\section{Evaluation of upper bounds for $R$ and $S$}\label{section:evaluation-RandS}

\begin{lemma}\label{lemma:evaluation-of-h}
For every $n\in \mathbb{N}$ and real number $c>1/q$, we have
\begin{gather} \label{inequality:h1}
    |h(n)|\leq (k-1)\left(q^{-1}\log_k n+1\right),\\ \label{inequality:h2}
    \sum_{n=1}^\infty \frac{|h(n)|}{n^{c}} \leq (k-1)\varphi(k,cq) \zeta(cq).
\end{gather}
\end{lemma}
\begin{proof}
We have \eqref{inequality:h1} immediately since the definition of $h(n)$ implies 
\[
|h(n)|\leq\sum_{\substack{d\geq0\\ k^{qd}\mid n}}(k-1)\leq \sum_{0\leq d\leq q^{-1}\log_k n}(k-1)\leq(k-1)\left(q^{-1}\log_k n+1\right).
\]
We also obtain \eqref{inequality:h2} easily since for every real number $c>1/q$, 
\[
\sum_{n=1}^\infty \frac{|h(n)|}{n^{c}} = \sum_{n=1}^\infty \frac{1}{n^c} \left\lvert \sum_{d \mid n  } f(d)g(n/d)\right\rvert \leq \left( \sum_{n=0}^\infty \frac{1}{k^{qnc}}\right)\left(\sum_{n=1}^\infty \frac{k-1}{n^{qs}}\right).
\]
\end{proof}

\begin{lemma}\label{lemma:Evaluation-R}
Let $x=k^{ql+p}$. We have 
\[
R\ll \frac{x\log x}{T}+\frac{x^c}{(cq-1)T}
\]
uniformly in $l\in\N$, $T\geq2$, $1/q<c<2$.
\end{lemma}
\begin{proof} Let 
\begin{align*}
    R_1\coloneqq\sum_{\substack{x/2<n<2x \\ n\not=x}}|h(n)|\min\left(\frac{x}{T|x-n|},1\right),\quad R_2:=\frac{x^c}{T}\sum_{n=1}^\infty\frac{|h(n)|}{n^c}.
\end{align*}
Then, by \eqref{inequality:R}, we have $R\ll R_1+R_2$. We first evaluate the upper bounds for $R_2$. Lemma~\ref{lemma:evaluation-of-h} leads to
\[
R_2\leq (k-1)\frac{x^c}{T}\varphi(k,cq)\zeta(cq)\ll\frac{x^c}{(cq-1)T}.
\]
For evaluating upper bounds for $R_1$, we see that 
\[
\frac{1}{T|1-n/x |}=1\iff n=x (1\pm T^{-1})=:U^{\pm}.
\]
Thus, by setting 
\[
R_{11}= \sum_{x/2<n\leq U^{-}}\left\lvert \frac{x h(n)}{x-n}\right\rvert, \quad 
R_{12}=\sum_{U^{+}\leq n<2x}\left\lvert \frac{x h(n)}{x-n}\right\rvert, \quad
R_{13}= \sum_{\substack{U^{-}<n<U^{+};\\ n\not=x}}|h(n)|,
\]
we have $R_1=(R_{11}+R_{12})/T +R_{13}$.  By the definition of $h$ and $x=k^{ql+p}\in\Z$, we obtain 
\begin{align*}
    R_{11}
    &=x\sum_{x/2<n\leq U^{-}}\frac{|h(n)|}{x- n}\ll x\sum_{1\leq n\leq x-1}\frac{1}{x- n}\sum_{\substack{d\geq0\\ k^{qd}\mid n}}1
    =x\sum_{\substack{d\geq0,m\geq1\\ 1\leq mk^{qd}\leq x-1}}\frac{1}{x-mk^{qd}}\\
    &=x\sum_{0\leq d\leq q^{-1}\log_k(x-1)}\frac{1}{k^{qd}} \sum_{1\leq m\leq\frac{x-1}{k^{qd}}}\frac{1}{x/k^{qd}- m}\\
    &=x\sum_{0\leq d\leq q^{-1}\log_k(x-1)}\frac{1}{k^{qd}}\sum_{1\leq m\leq \left\lfloor \frac{x}{k^{qd}} \right\rfloor-1}\frac{1}{x/k^{qd}- m},
\end{align*}
where we apply the propertiy $x/k^{qd}= k^{(l-d)q+p} \in \Z$ for every $0\leq d\leq l$ at the last equation. Therefore, we have
\[
R_{11}\ll x\sum_{0\leq d\leq q^{-1}\log_k(x-1)}\frac{\log (x/k^{qd})}{k^{qd}}  \ll x\log x.
\]
Similarly, we also obtain 
\begin{align*}
    R_{12}
    &\leq x \sum_{x+1\leq n<2x}\frac{1}{n- x}\sum_{\substack{d\geq0\\k^{qd}\mid n}}1
    =x\sum_{\substack{q^{-1}\log_k(x+1)\leq d\leq q^{-1}\log_k(2x)}}\frac{1}{k^{qd}} \sum_{x+1\leq mk^{qd}<2x}\frac{1}{m- x/k^{qd}}\\
    &\ll  x\sum_{\substack{q^{-1}\log_k(x+1)\leq d\leq q^{-1}\log_k(2x)}}\frac{\log(x/k^{qd}) }{k^{qd}} \ll x\log x.
\end{align*}
Furthermore, Lemma \ref{lemma:evaluation-of-h} implies  
\[
R_{13}=\sum_{\substack{U^{-}<n<U^{+}\\ n\not=x}}|h(n)|\leq(k-1)(q^{-1}\log_k x+1)\sum_{\substack{U^{-}<n<U^{+}\\ n\not=x}}1.
\]
Recall that $x\in \Z$; hence if $U^{+}-U^{-}<1$, the above sum on the most right-hand side is $0$. Therefore, we have
\[
\sum_{\substack{U^{-}<n<U^{+}\\ n\not=x}}1\leq U^{+}-U^{-}=x((1+T^{-1})-(1-T^{-1}))\ll\frac{x}{T}.
\]
Combining the upper bounds for $R_{11}, R_{12}, R_{13}$, we have $R_{1}\ll (x\log x)/T$, and hence
\[
R\ll R_1+R_2\ll \frac{x\log x}{T}+\frac{x^c}{(cq-1)T}.
\]
\end{proof}

\begin{corollary}\label{corollary:evaluation-R}
By choosing $c=1+1/\log x$ and $T\gg x\log x$,  we have $R\ll 1$.
\end{corollary}
Let us next give an upper bound for $S_1$. 

\begin{lemma}\label{lemma:convexity}
For every $t\geq t_0>0$ uniformly in $\sigma\in \mathbb{R}$,  
\begin{align*}
    \zeta(\sigma+it)\ll
    \begin{cases}
        1 & (\sigma\geq 2), \\
	\log t  &  (1\leq \sigma\leq 2),\\
	t^{\frac{1}{2}(1-\sigma)}\log t  &  (0\leq\sigma\leq 1 ),\\
	t^{\frac{1}{2}-\sigma}\log t  &  (\sigma\leq0).
    \end{cases}
\end{align*}
\end{lemma}
\begin{proof}
See \cite[Theorem~1.9]{Ivic}.
\end{proof}

\begin{lemma}\label{lemma:evaluation-S}
    Let  $\sigma_0$ be a constant depending only on $q$ satisfying $-\frac{1}{2(q-1)}<\sigma_0 <0$, where we define $-\frac{1}{2(q-1)}=-\infty$ if $q=1$. Let $x\geq 2$,  $T\in [2,\infty)\setminus\mathcal{T}_\delta$, and $c=1+1/\log x$. Assume that $T\asymp x\log x$.  Then,  we have $S_1\ll 1$ uniformly in such $x$ and $T$.
\end{lemma}

\begin{proof}
By the definition of $S_1$, it follows that
\begin{align*}
    S_1
    &=\frac{1}{2\pi i}\int_{\sigma_0}^c\varphi(k,q(\sigma+iT))\eta(q(\sigma+iT))\frac{x^{\sigma+iT}}{\sigma+iT}d\sigma\\
    &-\frac{1}{2\pi i}\int_{\sigma_0}^c\varphi(k,q(\sigma-iT))\eta(q(\sigma-iT))\frac{x^{\sigma-iT}}{\sigma-iT}d\sigma.
\end{align*}
Since $\varphi(k,q(\sigma+iT)),\varphi(k,q(\sigma-iT))\ll1$ for every $(\sigma,T)\in[\sigma_0,c]\times\R\setminus\mathcal{T}$, the Schwarz reflection principle of $\zeta(s)$ and $\eta(s)=(1-k^{1-s})\zeta(s)$ imply
\begin{equation}\label{inequality-S1-1}
    S_1 \ll \frac{1}{T}\int_{q\sigma_0}^{qc}|\zeta(\sigma+iqT)|x^{\frac{\sigma}{q}}d\sigma.
\end{equation}
Further, we decompose the right-hand side of \eqref{inequality-S1-1} into three integrals as
\begin{align}\label{equation:S1-2}
    \frac{1}{T}\left(\int_{[q\sigma_0,0]}+ \int_{[0,1]}+\int_{[1,qc]}\right) |\zeta(\sigma+iqT)|x^{\frac{\sigma}{q}}d\sigma \eqqcolon S_{11}+S_{12}+S_{13}. 
\end{align}
Lemma~\ref{lemma:convexity} implies that 
\begin{gather*}
    S_{11} \ll \displaystyle{\int_{q\sigma_0}^{0} T^{-\sigma-\frac{1}{2}}x^{\frac{\sigma}{q}}\log T d\sigma}, \quad
    S_{12} \ll  \displaystyle{\int_{0}^{1}T^{-\frac{1}{2}-\frac{\sigma}{2} }x^{\frac{\sigma}{q}}\log T d\sigma}, \quad
    S_{13} \ll \displaystyle{\int_{1}^{qc} T^{-1}x^{\frac{\sigma}{q}} \log T d\sigma}.
\end{gather*}
For $S_{13}$, by the choices of $T$ and $c$, we have $S_{13} \ll T^{-1} x^{c} \log T\ll 1$. For $S_{11}$, by using $T\asymp x\log x$, we see that
\begin{align*}
    S_{11}
    &\ll T^{-\frac{1}{2}} \log T \int_{q\sigma_0}^0\left(x^{\frac{1}{q}}T^{-1}\right)^{\sigma} d\sigma
    \ll x^{-\frac{1}{2}}\log^{\frac{1}{2}}x\int_{q\sigma_0}^0\left(x^{1-\frac{1}{q}}\log x\right)^{-\sigma}d\sigma\\
    &\ll x^{-\frac{1}{2}}\log^{\frac{1}{2}}x\cdot \frac{x^{-(q-1)\sigma_0}\log^{- q\sigma_0}x}{\log x}
    =x^{-\frac{1}{2}-(q-1)\sigma_0}(\log x)^{-\frac{1}{2}-q\sigma_0}.
\end{align*}
The most right-hand side is $\ll1$ since $-\frac{1}{2(q-1)}<\sigma_0$. For $S_{12}$, in a similar manner, 
\[
S_{12} \ll x^{-\frac{1}{2}}(\log x)^{\frac{1}{2}} \int_{0}^1 \left( x^{\frac{1}{2}-\frac{1}{q}} \log^{-\frac{1}{2}} x  \right)^{-\sigma} d \sigma 
\]
In the case $q=1$, since $1/2-1/q=-1/2$,  we have 
\[
S_{12} \ll x^{-\frac{1}{2}} (\log x)^{\frac{1}{2}} \frac{x^{\frac{1}{2}} (\log x)^{\frac{1}{2}}}{\log x}\ll 1.
\]
In the case $q\geq 2$, since $1/2-1/q\geq 0$, we also obtain $S_{12}\ll x^{-\frac{1}{2}}\log x \ll 1$, completing the proof of Lemma~\ref{lemma:evaluation-S}.  
\end{proof}

\begin{proposition}\label{proposition:perronRandS}
Let $\sigma_0$ be as in Lemma~\ref{lemma:evaluation-S}. Then, we have  
\begin{align*}
    A(l)
    &=-\frac{1}{2\pi i }\sum_{0<|n|\leq \frac{q\log k}{2\pi} T} \zeta\left(\frac{2n\pi i}{\log k}\right) \frac{e^{2n\pi i p/q}}{n}  +S_{\sigma_0}(l,T)+O(1)+
    \begin{cases}
         0  &  \quad\textup{if } p=q=1,\\
        \frac{l}{2}  &  \quad\textup{otherwise}
    \end{cases}
\end{align*}
uniformly in $l\in \mathbb{N}$ and $T\geq 2$ with $T\asymp lk^{ql}$, where  
\[
S=S_{\sigma_0}(l,T)\coloneqq\frac{1}{2\pi i(k-1)}\int_{\sigma_0-iT}^{\sigma_0+iT}\eta(qs)\frac{k^{(ql+p)s}}{1-k^{-qs}}\frac{ds}{s}.
\]
\end{proposition}
\begin{proof}
By Collorary~\ref{corollary:evaluation-R} and Lemma~\ref{lemma:evaluation-S}, the equation \eqref{equation:PerronResidue} implies that
\begin{equation}\label{equation:PerronResidue2}
    \sum_{n\leq x}{}^\prime h(n)=(k-1)\frac{\log x}{2q\log k}+\frac{1}{q\log k}\sum_{0<|s_n|\leq T}\frac{\eta(qs_n)}{s_n}x^{s_n}+S_0+O(1)
\end{equation}
for $T\in [2,\infty)\setminus \mathcal{T}$ and $T\asymp x\log x$. By substituting $x=k^{ql+p}\in \mathbb{Z}$, the equation \eqref{equation:numberof1} leads to 
\[
\sum_{n\leq x}{}^\prime h(n)= \sum_{n\leq x}h(n) -\frac{h(k^{ql+p})}{2}+O(1)= (k-1)A(l)- \frac{h(k^{ql+p})}{2} + O(1).
\]
In the case $q=1$, then $p=1$ and we obtain 
\[
h(k^{ql+p})= h(k^{l+1})= \sum_{\substack{d\geq 0\\ k^d\mid k^{l+1}} } g(k^{(l-d)+1}) =  \sum_{0\leq d\leq l+1}  b(k^{(l-d)+1})=(1-k)l+1.
\]
When $q\geq 2$, we recall that $\gcd(p,q)=1$, and $k$ is not a $q$-th power. Therefore, by combining the definitions of $h(n)$ and $g(n)$, we obtain
\[
h(k^{ql+p})=  \sum_{\substack{d\geq 0\\ k^{qd}\mid k^{ql+p}} } g(k^{q(l-d)+p}) =  \sum_{0\leq d\leq l}  g(k^{q(l-d)+p})=0. 
\]
Thus, the left-hand side of \eqref{equation:PerronResidue2} is 
\begin{align*}
    (k-1)A(l)+O(1)+
    \begin{cases}
        (k-1)\frac{l}{2}  &  \quad\textup{if } p=q=1,\\
        0  &  \quad\textup{otherwise}.
    \end{cases}
\end{align*}
Further, recalling that $s_n=2n\pi i/(q\log k)$ and $x=k^{ql+p}$, we have 
\begin{gather*}
    1-k^{1-qs_n}= 1-k,\quad x^{s_n}=k^{s_n(ql+p)}=e^{2n\pi i  p/q},
\end{gather*}
and hence 
\[
\frac{1}{q\log k}\sum_{0<|s_n|\leq T}\frac{\eta(qs_n)}{s_n}x^{s_n}= -\frac{k-1}{2\pi i }\sum_{0<|n|\leq \frac{q\log k}{2\pi} T} \zeta\left(\frac{2n\pi i}{\log k}\right) \frac{e^{2n\pi i p/q}}{n}.  
\]
Therefore, we have
\begin{align*}
    A(l) = -\frac{1}{2\pi i }\sum_{0<|n|\leq \frac{q\log k}{2\pi} T} \zeta\left(\frac{2n\pi i}{\log k}\right) \frac{e^{2n\pi i p/q}}{n}+ \frac{S_0}{k-1}+O(1)+
\begin{cases}
    0  &  \quad\textup{if } p=q=1,\\
    \frac{l}{2}  &  \quad\textup{otherwise}
\end{cases}
\end{align*}
for $T\in [2,\infty)\setminus \mathcal{T}$ and $T\asymp x\log x$. We can remove the condition $T\notin \mathcal{T}$. Indeed, for sufficiently large $T\geq 2$,  we observe that
\begin{align*}
    \sum_{{T<|s_n|\leq T+1}}\frac{\eta(qs_n)}{qs_n}x^{s_n}\ll\frac{|qs_n|^{\frac{1}{2}+\epsilon}}{T}\ll T^{-\frac{1}{2}+\epsilon}.
\end{align*}
In addition,  if $T$ satisfies $T\asymp x\log x$, then
\begin{align*}
    &\int_{\sigma_0+iT}^{\sigma_0+i(T+1)}\varphi(k,qs)\eta(qs)\frac{x^s}{s}ds\ll x^{\sigma_0}\int_T^{T+1} \left\lvert \frac{\zeta(q\sigma_0+iqt)}{\sigma_0+it}\right\rvert dt\\
    &\ll x^{\sigma_0}\int_T^{T+1}t^{-\frac{1}{2}-q\sigma_0}dt
    \ll x^{\sigma_0}T^{-\frac{1}{2}-q\sigma_0}
    \ll x^{-\frac{1}{2}-(q-1)\sigma_0}(\log x)^{-\frac{1}{2}-q\sigma_0}\ll 1,
\end{align*}
where $-\frac{1}{2(q-1)}<\sigma_0$ leads to the last inequality.  Therefore, we conclude Proposition~\ref{proposition:perronRandS}.
\end{proof}

\section{Applying the functional equation to $S$ and the proof of Theorem~\ref{theorem:main2}}\label{section:evaluation-S}

This section gives proofs of Theorem~\ref{theorem:main2} and the following theorem.

\begin{theorem}\label{theorem:shift-trick} Suppose $1\leq p<q$ and $\gcd (p,q)=1$. For every integer $l\geq 2$ and real number $T\asymp lk^{ql}$,  we have 
\[
A(\log_k T )= \frac{\log_k T}{2} -\frac{1}{2\pi i }\sum_{0<|n|\leq \frac{q\log k}{2\pi} T} \zeta\left(\frac{2n\pi i}{\log k}\right) \frac{e^{2n\pi i p/q}}{n}+\sum_{0\leq m\leq \log_k(Tk^{-l}) } E_m(l)   +O(1),
\]
where $E_m(l)$ satisfies 
\begin{equation}\label{inequality:Em(l)} 
    |E_m(l)| \leq \min \left(\frac{ 1}{2},\ \frac{B}{  l k^{(q-1)l -m} |\sin(\pi k^{m+l+p/q} )| } \right)
\end{equation}
for some constant $B>0$ depending only on $k$, $p$, and $q$.
\end{theorem}

\begin{lemma}[the functional equation of the Riemann zeta function]\label{Lemma:FunctionalEquation} For every $s\in \mathbb{C}\setminus \{1\}$, we have
$\zeta (s) =\chi (s) \zeta(1-s)$, where $\chi(s)=2^{s-1}\pi^s\sec(\pi s/2)/\Gamma(s)$. Further, for any fixed $\sigma\in \mathbb{R}$ and for $t\geq1$, we have
\[
\chi(s)= \left(2\pi/ t \right)^{\sigma+it-1/2}e^{i (t+\pi/4)}\left(1+O\left(\frac{1}{t}\right)\right)
\]
\end{lemma}

\begin{proof}
See \cite[(2.1.8), (4.12.3)]{Titchmarsh}.
\end{proof}

\begin{lemma}\label{lemma:FunctionalEquation2} Let $s=\sigma+it$. Fix any $\sigma<0$. For $t\geq 1$, we have  
\[
\frac{\zeta(qs)}{s}=e^{-i\pi /4}\cdot \left(\frac{2\pi}{q}\right)^{q\sigma-1/2} \cdot t^{-1/2-q\sigma} \sum_{n=1}^\infty n^{q\sigma-1} \left(\frac{2n\pi e}{qt} \right)^{qit}+ O(t^{-3/2-q\sigma}).
\]
\end{lemma}
\begin{proof} Lemma~\ref{Lemma:FunctionalEquation} implies that for $t\geq 1$, 
\begin{align*}
    \frac{\zeta(qs)}{s} &= \frac{1}{it}\left(\frac{1}{1+\sigma/(it) } \right) \cdot \left(\frac{2\pi}{qt} \right)^{q\sigma+qit-1/2}e^{i (qt+\pi/4)} \zeta(1-q\sigma-qit)  \left(1+O\left(\frac{1}{t}\right)\right) \\
    &= \frac{1}{it}  \cdot \left(\frac{2\pi}{qt} \right)^{q\sigma+qit-1/2}e^{i (qt+\pi/4)} \zeta(1-q\sigma-qit)  \left(1+O\left(\frac{1}{t}\right)\right)\\
    &= e^{-i\pi /4}\cdot \left(\frac{2\pi}{q}\right)^{q\sigma-1/2} \cdot t^{-1/2-q\sigma} \sum_{n=1}^\infty n^{q\sigma-1} \left(\frac{2n\pi e}{qt} \right)^{qit}\left(1+O\left(\frac{1}{t}\right)\right)\\
    &=e^{-i\pi /4}\cdot \left(\frac{2\pi}{q}\right)^{q\sigma-1/2} \cdot t^{-1/2-q\sigma} \sum_{n=1}^\infty n^{q\sigma-1} \left(\frac{2n\pi e}{qt} \right)^{qit}+ O(t^{-3/2-q\sigma}).
\end{align*}
\end{proof}

\begin{lemma}\label{Lemma:saw-tooth-function}
Let $\psi(y)$ be the saw-tooth function, that is, 
\[
\psi(y)= 
\begin{cases}
    \{y\}-1/2  &  \textup{if } y\notin \mathbb{Z},\\
    0  &  \textup{if } y\in \mathbb{Z}.
\end{cases}
\]
Then for every $K\in \mathbb{N}$ and $y\in\R$, we have 
\[
\left\lvert \sum_{k=1}^K \frac{\sin (2\pi ky)}{\pi k}+\psi(y)\right\rvert \leq \min \left(\frac{1}{2}, \frac{1}{(2K+1)\pi |\sin \pi y|}\right).
\]
\end{lemma}
\begin{proof}
See \cite[Lemma~D.1]{Montgomery-Vaughan}.
\end{proof}
We take any integer $l\geq 2$ and any real number $T\asymp lk^{ql}$.
\begin{lemma}\label{lemma:evaluationS-step1}
Let $(a_m)_{m\geq 0}$ be the sequence  in \eqref{equation:geoseries}. For every $m\in \mathbb{Z}_{\geq 0}$ and $n\in \mathbb{N}$, let $\alpha_{m,n}=\alpha_{m,n}(l)=k^{l+m+p/q}n \pi$. For every $t\in [1,T]$, we define
\[
F(t)=F_{m,n}(t)\coloneqq qt \log \left(\frac{2k^{m+l+p/q}n\pi e }{qt} \right) =qt\log\left(\frac{2\alpha_{m,n}e}{qt} \right).   
\]
Then we have 
\begin{equation}\label{equation:ssigma1}
    (k-1)S_{-1/(2q)}(l,T) =\frac{1}{2\pi^{3/2}} \sum_{n=1}^\infty \sum_{m=0}^\infty n^{-1} a_m \alpha_{m,n}^{-1/2} q \cdot \Re \left(e^{-i\pi/4}\int_{1}^{T} e^{iF_{m,n}(t)} dt\right)+O(1).
\end{equation}
\end{lemma}

\begin{proof}

For every $\Re(s)<0$, by \eqref{equation:geoseries}, we recall that
\begin{align*}
    &\frac{1-k^{1-qs}}{1-k^{-qs}}= \sum_{m=0}^\infty a_m k^{mqs}.
\end{align*}
 We now choose $\sigma=-1/(2q)$. Then  for $s=\sigma+it$ ($t\in [-1,1]$), we have 
\[
\frac{\zeta(qs)}{s}\cdot\frac{1-k^{1-qs}}{1-k^{-qs}}k^{(ql+p)s}\ll_{k,p,q} k^{-l/2}\ll 1.  
\]
By applying Lemma~\ref{lemma:FunctionalEquation2}, for $s=\sigma+it$ ($t\in [1,T]$),  we have
\begin{align*}
    &\frac{\zeta(qs)}{s}\cdot\frac{1-k^{1-qs}}{1-k^{-qs}}k^{(ql+p)s}\\
    &= \sum_{n=1}^\infty \sum_{m=0}^\infty  a_m k^{mqs} k^{(ql+p)s}e^{-i\pi /4}\cdot \left(\frac{2\pi}{q}\right)^{-1}  n^{-3/2} \left(\frac{2n\pi e}{qt} \right)^{qit}+ O(k^{-l/2} t^{-1})\\
    &= \frac{1}{2\pi} \sum_{n=1}^\infty \sum_{m=0}^\infty n^{-3/2} a_m k^{-(m+l+p/q)/2}q \cdot e^{-i\pi /4} \left(\frac{2k^{m+l+p/q}n\pi e}{qt} \right)^{qit}+ O(k^{-l/2} t^{-1})\\
    &=\frac{1}{2\pi} \sum_{n=1}^\infty \sum_{m=0}^\infty n^{-1} a_m \alpha_{m,n}^{-1/2}q \cdot e^{-i\pi /4} \left(\frac{2\alpha_{m,n}e}{qt} \right)^{qit}+ O(k^{-l/2} t^{-1}),
\end{align*}
and hence the Schwarz reflection principle leads to the following:
\begin{align*}
    (k-1)S
    &=(k-1)S_\sigma(l,T) \\
    &=\frac{1}{2\pi i}\int_{\sigma-iT}^{\sigma+iT}\frac{\zeta(qs)}{s}\cdot\frac{1-k^{1-qs}}{1-k^{-qs}}k^{(ql+p)s}ds\\
    &=\frac{1}{2\pi^2} \sum_{n=1}^\infty \sum_{m=0}^\infty n^{-1} a_m \alpha_{m,n}^{-1/2} q  \cdot \Re \left(e^{-i\pi/4}\int_{1}^{T} \left(\frac{2\alpha_{m,n}e}{qt}\right)^{qit} dt\right).\\
    &\quad +O(1)+O\left(k^{-l/2}\int_1^Tt^{-1}dt\right).
\end{align*}
Since $T\ll lk^{ql}$, it follows that
\[
k^{-l/2}\int_1^Tt^{-1} dt \ll k^{-l/2} \log T \ll 1.
\]
By the definitions of  $\alpha_{m,n}$ and $F_{m,n}(t)$, we obtain Lemma~\ref{lemma:evaluationS-step1}. 
\end{proof}

Let $m\in \mathbb{Z}_{\geq 0}$ and $n\in \mathbb{N}$. Let $c=c_{m,n}=2\alpha_{m,n}/q=2k^{m+l+p/q}n\pi/q$. Since $F'(t)= q\log(2\alpha_{m,n}/(qt))$, we have $F'(t)=0$ if and only if $t=c$. Let $\delta>0$ be a sufficiently small absolute constant. It is enough to choose $\delta=1/4$. We define
\begin{equation}
    U=U_{m,n}= \min (1/2, \alpha_{m,n}^{-1/2} (m+l)^{1+\delta} n^{\delta}).
\end{equation}
Then, we have
\begin{equation}\label{inequality:convergentU} 
    \sum_{n=1}^\infty \sum_{m=0}^\infty n^{-1} \alpha_{m,n}^{-1/2} U_{m,n}^{-1} \ll 1.   
\end{equation}
Let us decompose the sum on the right-hand side of \eqref{equation:ssigma1} into three sums as follows:
\begin{align*}
    &\sum_{n=1}^\infty \sum_{m=0}^\infty n^{-1} a_m \alpha_{m,n}^{-1/2} q \cdot \Re \left(e^{-i\pi/4}\int_{1}^{T} e^{iF_{m,n}(t)} dt\right)\\
    &= \underset{T\leq c_{m,n}(1-U_{m,n}) }{\sum_{n=1}^\infty \sum_{m=0}^\infty}  + \underset{c_{m,n}(1-U_{m,n})<T<c_{m,n}(1+U_{m,n}) }{\sum_{n=1}^\infty \sum_{m=0}^\infty}  + \underset{c_{m,n}(1+U_{m,n})\leq T }{\sum_{n=1}^\infty \sum_{m=0}^\infty}  \eqqcolon S_{01}+S_{02}+S_{03}. 
\end{align*}

\begin{lemma}\label{lemma:S01}
We have $S_{01}\ll 1$. 
\end{lemma}

\begin{proof}
Take any $(m,n)\in \mathbb{Z}_{\geq 0}\times \mathbb{N}$ with $T \leq c_{m,n}(1-U_{m,n})$. Then, each $t\in [1,T]$ satisfies
\[
|F'(t)| \geq q\log \left(\frac{1}{1-U} \right) \gg U.
\]
Therefore, by Lemma~\ref{lemma:exponential-sums}, we obtain 
\begin{equation}\label{inequality:first-deriv}
    \left\lvert \int_{1}^{T} e^{iF_{m,n}(t)} dt\right\rvert \ll U^{-1},
\end{equation}
and hence \eqref{inequality:convergentU} implies that 
\begin{align*}
    S_{01}= \underset{T\leq c_{m,n}(1-U_{m,n}) }{\sum_{n=1}^\infty \sum_{m=0 }^\infty} n^{-1} a_m \alpha_{m,n}^{-1/2} q \cdot \Re \left(e^{-i\pi/4}\int_{1}^{T} e^{iF_{m,n}(t)} dt\right) \ll 1. 
\end{align*}
\end{proof}

\begin{lemma}\label{lemma:S02}
We have $S_{02}\ll 1$. 
\end{lemma}

\begin{proof}
We discuss the case $c(1-U)<T\leq c(1+U)$. Let 
\[
J=\{(m,n)\in \Z_{\geq 0} \times \N \colon T/(1+U_{m,n})\leq c_{m,n} <T/(1-U_{m,n})\}.
\]
Take any $(m,n)\in J$. Then, the following inequalities hold: 
\begin{enumerate}
\item \label{f0}$c_{m,n}\asymp T$;
\item \label{f1}$m\leq (q-1)l +\log_k l+O(1)=\log_k (Tk^{-l})+O(1)$;
\item \label{f2}$n \asymp Tk^{-(l+m)}$;
\item \label{f3}$\alpha_{m,n} \asymp T$.
\end{enumerate}
Indeed, \eqref{f0} is trivial by $U\leq 1/2$ and  the choice of $(m,n)$. Also, \eqref{f1} immediately follows from  $k^{m+l}\ll c_{m,n}\asymp T$ and $T\asymp l k^{ql}$. In addition, \eqref{f0} and the definition of $c_{m,n}$ imply \eqref{f2} and \eqref{f3}.  

By the choice of $U_{m,n}$, \eqref{f1}, \eqref{f2}, and \eqref{f3}, we obtain 
\[
U_{m,n} \asymp \min( 1/2, T^{-1/2+\delta} l^{1+\delta} k^{-\delta(m+l)} ).
\]
 Further, by the definition of $J$, $n$ satisfies 
\begin{equation}\label{inequality:range-of-n}
    \frac{qT}{2k^{m+l+p/q}\pi (1+U) } \leq n < \frac{qT}{2k^{m+l+p/q}\pi (1-U) }.
\end{equation}
The number of $n$'s satisfying \eqref{inequality:range-of-n} is at most $\ll 1+ UTk^{-m-l}\ll 1+ T^{1/2+\delta} l^{1+\delta} k^{-(1+\delta)(m+l)}$. Similarly to \eqref{inequality:first-deriv}, we obtain 
\[
\int_{1}^{c(1-U)} e^{iF_{m,n}(t)} dt\ll U^{-1}. 
\]
Therefore, by  \eqref{f1}, \eqref{f2}, and \eqref{f3}, we have
\begin{align}\label{inequality:case-J}
    S_{02}&= \underset{(m,n)\in J }{\sum_{n=1}^\infty \sum_{m=0 }^\infty} n^{-1} a_m \alpha_{m,n}^{-1/2} q \cdot \Re \left(e^{-i\pi/4}\int_{1}^{T} e^{iF_{m,n}(t)} dt\right) \\ \nonumber
    &\ll 1+ \sum_{m \leq (q-1)l+\log_k l+O(1)}  T^{-3/2} k^{m+l} \left\lvert \sum_{\substack {n\in \N \text{ with }\eqref{inequality:range-of-n}  \\ (m,n)\in J} }  \int_{c(1-U)}^{T} e^{iF_{m,n}(t)} dt \right\rvert. 
\end{align}
For every $t\in [c(1-U),T]$, we observe that 
\[
|F''_{m,n}(t)|= |q/t|\gg T^{-1}.  
\]
Therefore, Lemma~\ref{lemma:second-derivative} with $F\coloneqq F_{m,n}$ yields that the most right-hand side of \eqref{inequality:case-J} is 
\begin{align*}
    &\ll 1+\sum_{m \leq (q-1)l+\log l+O(1)}  T^{-3/2} k^{m+l} (1+ T^{1/2+\delta} l^{1+\delta} k^{-(1+\delta)(m+l)}) T^{1/2}.
\end{align*}
By simple calculation and $T\asymp lk^{ql}$, the above is 
\begin{align*}
    &\ll 1+\sum_{m \leq (q-1)l+\log_k l+O(1)} (T^{-1}  k^{m+l}+ T^{-1/2+\delta }  l^{1+\delta} k^{-\delta(m+l)} ) \\
    &\ll 1+ T^{-1}T + T^{-1/2}l^{1+\delta} \ll 1.  
\end{align*}
\end{proof}

\begin{lemma}\label{lemma:S03}
We have
\[
\frac{S_{03}}{2\pi^{3/2}}  = \sum_{m\leq \log_k(Tk^{-l})+O(1)}a_m \sum_{\substack{n\in \mathbb{N}\\ c_{m,n}(1+U_{m,n})<T }} \frac{\sin (2\alpha_{m,n} )}{n\pi } +O(1 ).    
\]

\end{lemma}
\begin{proof}
Take any $(m,n)\in\Z_{\geq0}\times \N$ with $T>c_{m,n}(1+U_{m,n})$. We have $F'(c_{m,n})=0$ and $c_{m,n}\in [1,T]$. Further, we also get 
\begin{equation}\label{equation:upper bound of m}
    m\leq (q-1)l +\log_k l+O(1)=\log_k (Tk^{-l})+O(1)
\end{equation}
since $k^{m+l}\ll c_{m,n} \ll lk^{ql}$. To apply Lemma~\ref{lemma:stationary-phase}, check \eqref{condition:stationary1} and  \eqref{condition:stationary2}. It follows that 
\[
F''(t) =  -q/t,\quad \text{and} \quad \quad F'''(t)= q/t^2. 
\]
Therefore, for every $t\in [c(1-U),c(1+U)]$, we obtain 
\[
|F''(t) |  \asymp c^{-1}, \quad \text{and} \quad |F'''(t)| \ll c^{-2}.
\]
In addition, $|F'(c(1-U) )|$, $|F'(c(1+U))|\gg U$. Thus, by $U\gg c^{-1/2}$, Lemma~\ref{lemma:stationary-phase} leads to 
\begin{align} \nonumber
    \int_{c(1-U)}^{c(1+U)} e^{iF(t)}  dt &=(2\pi)^{\frac{1}{2}} \frac{e^{-\pi i/4 +iF(c)} }{|F''(c)|^{1/2}}  +O(c^{\frac{4}{5}}c^{-\frac{2}{5}} )+O\left(\min\left(U^{-1},c^{\frac{1}{2}} \right)\right) \\ \label{equation:integral-main}
    &= 2\pi^{\frac{1}{2}} e^{-\pi i/4} e^{2i\alpha_{m,n} } \alpha_{m,n}^{1/2}/q + O(c^{\frac{2}{5}}) +O(U^{-1}).
\end{align}
By Lemma~\ref{lemma:exponential-sums}, we obtain 
\begin{equation}\label{Inequality-exponential-2}
    \left\lvert \int_{c(1+U)}^T e^{iF(t)} dt \right\rvert \ll U^{-1} . 
\end{equation}
Combining \eqref{inequality:convergentU}, \eqref{equation:integral-main} and \eqref{Inequality-exponential-2}, we have 
\begin{align*}
    \frac{S_{03}}{2\pi^{3/2}}  &=\frac{1}{2\pi^{3/2}} \underset{c_{m,n}(1+U_{m,n}) <T }{\sum_{n=1}^\infty \sum_{m=0 }^\infty} n^{-1} a_m \alpha_{m,n}^{-1/2} q \cdot \Re \left(e^{-i\pi/4}\int_{1}^{T} e^{iF_{m,n}(t)} dt\right) \\
    &= \frac{1}{2\pi^{3/2}} \underset{c_{m,n}(1+U_{m,n}) <T }{\sum_{n=1}^\infty \sum_{m=0 }^\infty} n^{-1} a_m \alpha_{m,n}^{-1/2} q \cdot  \Re \left(\frac{2\pi^{\frac{1}{2}}}{i} e^{2i\alpha_{m,n}} \alpha_{m,n}^{1/2}/q  \right)+O(1) \\
    &= \sum_{m\leq \log_k(Tk^{-l})+O(1)}a_m \sum_{\substack{n\in \mathbb{N}\\ c_{m,n}(1+U_{m,n})<T }} \frac{\sin (2\alpha_{m,n} )}{n\pi } +O(1) ,  
\end{align*}
where we apply \eqref{equation:upper bound of m} to restrict the range of the summation.
\end{proof}

\begin{proof}[Proof of Theorem~\ref{theorem:main2}] Let $k$ be an integer not less than $2$ and let $p=q=1$. Take an arbitrary real number $l'\geq 2$.  Let $l$ be the integer satisfying $l\leq  l'-\log_k l'  <l+1 $. Choose $T= \frac{2\pi }{\log k} k^{l'}$. Then it follows that $T \asymp k^{l'} = l' k^{l'-\log_k l'} \asymp l k^{l}$. By combining Proposition~\ref{proposition:perronRandS} with $\sigma_0=-1/2$ and Lemmas~\ref{lemma:evaluationS-step1} to \ref{lemma:S03}, we obtain 
\begin{align*}
    0&=\sum_{0\leq m\leq l}\{k^{m+1}\}=A(l)\\
    &= -\frac{1}{2\pi i } \sum_{0<|n|\leq \frac{\log k}{2\pi} T } \zeta \left(\frac{2n\pi i}{\log k}\right) \frac{1}{n} \\
    &\hspace{30pt}+ \sum_{m\leq \log_k(Tk^{-l}) +O(1)} \frac{a_m}{k-1} \sum_{\substack{n\in \mathbb{N}\\ c_{m,n}(1+U_{m,n})<T}} \frac{\sin (2\alpha_{m,n})}{n\pi} +O(1).   
\end{align*}
We have $\sin(2\alpha_{m,n})=0$ since $\alpha_{m,n}=\pi k^{l+m+1}n\in \pi \mathbb{Z}$ for all integers $m\geq 0$ and $n\geq 1$.  Therefore, we obtain
\[
0= -\frac{1}{2\pi i} \sum_{0<|n|\leq k^{l'}} \zeta \left(\frac{2n\pi i}{\log k}\right) \frac{1}{n} +O(1), 
\]
which completes the proof of Theorem~\ref{theorem:main2}.
\end{proof}

\begin{proof}[Proof of Theorem~\ref{theorem:shift-trick}]
Take arbitrary integers $k$, $p$, and $q$ satisfying $k\geq 2$, $1\leq p<q$, and $\gcd(p,q)=1$. 
By Lemmas~\ref{lemma:evaluationS-step1} to \ref{lemma:S03}, we obtain 
\[
(k-1)S_{-1/(2q)}(l,T) =  \sum_{m\leq \log_k(Tk^{-l}) +O(1)}a_m \sum_{\substack{n\in \mathbb{N}\\ c_{m,n}(1+U_{m,n})<T }} \frac{\sin (2\alpha_{m,n} )}{n\pi } +O(1). 
\]
Lemma~\ref{Lemma:saw-tooth-function} with $y=k^{m+l+p/q}$ implies that
\[
\sum_{\substack{n\in \mathbb{N}\\ c_{m,n}(1+U_{m,n})<T }} \frac{\sin (2\alpha_{m,n} )}{n\pi } = -\psi(k^{m+l+p/q} ) + E_{m}(l),
\]
where  $E_m(l)$ satisfies
\begin{gather*}
    |E_{m}(l)| \leq \min \left( \frac{1}{2}, \frac{1}{(2K_m+1)|\sin (\pi k^{m+l+p/q})| } \right),\\ 
    K_m\coloneqq \max\{n\in \N \colon c_{m,n}(1+U_{m,n})< T\}. 
\end{gather*}
Since $c_{m,n}=2k^{l+m+p/q}n\pi /q$, $T \asymp lk^{ql}$, and $U\leq 1/2$, we get $K_m \gg lk^{(q-1)l -m }$, leading to \eqref{inequality:Em(l)}. Therefore, $(k-1)S_{-1/(2q)}(l,T)$ is 
\begin{align*}
    &= \sum_{0\leq m \leq \log_k(Tk^{-l}) +O(1)} (k-1) \left(\frac{1}{2} - \{k^{m+l+p/q} \} +E_m(l)\right) + O(1)\\
    &= (k-1)\left(\frac{ \log_k(Tk^{-l})}{2} -\sum_{0\leq m\leq \log_k(Tk^{-l})} \{k^{m+l+p/q} \}+\sum_{0\leq m\leq \log_k(Tk^{-l})} E_m(l) \right)+O(1).
\end{align*}
By Proposition~\ref{proposition:perronRandS} with $\sigma_0=-1/(2q)$, we have
\begin{align*}
    A(l)
    &=\frac{l} {2}-\frac{1}{2\pi i }\sum_{0<|n|\leq \frac{q\log k}{2\pi} T} \zeta\left(\frac{2n\pi i}{\log k}\right) \frac{e^{2n\pi i p/q}}{n} + \frac{ \log_k(Tk^{-l})}{2}\\ 
    &\hspace{30pt} -\sum_{0\leq m\leq \log_k(Tk^{-l})} \{k^{m+l+p/q} \}+\sum_{0\leq m\leq \log_k(Tk^{-l})} E_m(l)+O(1),
\end{align*}
which completes the proof of Theorem~\ref{theorem:shift-trick} since 
\[
A(l)+ \sum_{0\leq m\leq \log_k (Tk^{-l})} \{k^{m+l+p/q}\} = A(\log_k T )+O(1).
\]
\end{proof}

\section{Ridout's theorem and the completion of the proof}\label{section:completion}

\begin{theorem}[Ridout's theorem]\label{theorem:ridout}
Let $\alpha$ be any algebraic number other than 0; let $P_1,\ldots, P_s$, $Q_1,\ldots, Q_t$ be distinct primes; and let $\mu$, $\nu$, and $c$ be real numbers satisfying
\[
0\leq \mu \leq 1,\quad 0\leq \nu \leq 1,\quad c>0 . 
\]
Let $a$ and $b$ be restricted to integers of the form
\[
a=a^{*}P_1^{\rho_1}\cdots P_s^{\rho_s},\quad b = b^{*}Q_1^{\sigma_1}\cdots Q_t^{\sigma_t},
\]
where $\rho_1,\ldots, \rho_s$, $\sigma_1,\ldots, \sigma_t$ are non-negative integers and $a^*$, $b^*$ are integers satisfying
\[
0<a^* \leq ca^{\mu},\quad 0<b^* \leq cb^{\nu}.
\]
Then if $\kappa>\mu+\nu$, the inequality $0 <| \alpha-a/b|< b^{-\kappa}$ has only a finite number of solutions in $a$ and $b$.
\end{theorem}
\begin{proof}
See \cite{Ridout}.
\end{proof}

We have the following corollary by substituting $\mu=1$,  $\nu=0$, and $c=1$.
\begin{corollary}\label{Cororally:Ridout}Let $\alpha$ be any algebraic number other than 0. 
Let $Q_1,\ldots , Q_t$ be distinct primes. Let $b$ be an integer of the form 
\begin{equation}\label{equation:primefactors}
    b=Q_1^{\sigma_1}\cdots Q_{t}^{\sigma_t},
\end{equation}
where $\sigma_1,\ldots, \sigma_t$ are non-negative integers. Then for any $\epsilon>0$, there exists $C>0$ such that $|\alpha-a/b |\geq Cb^{-1-\epsilon}$
for every $a\in \mathbb{Z}$ and $b$ of the form  \eqref{equation:primefactors}. 
\end{corollary}

\begin{proof}[Proof of Theorem~\ref{theorem:generalization}]
Let $\gamma>0$ be an arbitrarily small constant. Let $l'$ be a sufficiently large real number. Take a positive integer $l$ such that $ql \leq l' -\log_k l'  < ql+q$. Choose $T=\frac{2\pi}{q\log k}k^{l'}$. Then we obtain
\[
T= \frac{2\pi}{q\log k}k^{l'} = \frac{2\pi}{q\log k}l' k^{l'-\log_k l' } \asymp lk^{ql}.    
\]
Since we have $A(l')= A(ql)+O(\log_k l')$, Theorem~\ref{theorem:shift-trick} leads to
\begin{align*}
    A(l')
    &=A(ql)+O(\log l')\\
    &= \frac{ql}{2} -\frac{1}{2\pi i }\sum_{0<|n|\leq k^{l'}} \zeta\left(\frac{2n\pi i}{\log k}\right) \frac{e^{2n\pi i p/q}}{n}+\sum_{0\leq m\leq (q-1)l} E_m(l)   +O(\log l'),
\end{align*}
where $E_{m}(l)$ satisfies \eqref{inequality:Em(l)}. Since $|\sin \pi x|\gg \|x\|$, for every $0\leq m\leq (q-1)l$, we have 
\[
E_{m}(l)\leq \min \left( \frac{1}{2}, \frac{C}{lk^{(q-1)l-m}\|k^{m+l}\cdot k^{p/q} \|}\right)
\]
for some constant $C>0$. By substituting $\alpha:=k^{p/q}$, $b:=k^{m+l}$, and  $\epsilon:=\gamma$ in Corollary~\ref{Cororally:Ridout}, we obtain 
\begin{equation}\label{inequality:Ridout}
    \|k^{m+l}\cdot k^{p/q}\|\gg_\gamma k^{-\gamma (m+l)} .
\end{equation}
Therefore, the inequality \eqref{inequality:Ridout} yields that 
\[
\sum_{0\leq m \leq \frac{(q-1-\gamma)l}{1+\gamma} } E_{m}(l) \ll_\gamma \sum_{0\leq m \leq \frac{(q-1-\gamma)l}{1+\gamma}} l^{-1} k^{(\gamma+1)m-(q-1-\gamma)l  } \ll_\gamma 1.
\]
Also, we obtain 
\[
\sum_{\frac{(q-1-\gamma)l}{1+\gamma}<m \leq (q-1)l} E_{m}(l)\ll \gamma ql, 
\]
where the implicit constant does not depend on $\gamma$. Therefore, we have
\[
\sum_{0\leq m\leq (q-1)l} E_{m}(l)=O_\gamma(1) +O(\gamma l ).
\]
By combining the above discussion, we obtain 
\[
A(l')= \frac{l'}{2} -\frac{1}{2\pi i }\sum_{0<|n|\leq k^{l'}} \zeta\left(\frac{2n\pi i}{\log k}\right) \frac{e^{2n\pi i p/q}}{n}+O(\gamma l') +O_\gamma(1) +O(\log l'),
\]
which implies that 
\[
\lim_{l'\to \infty} \frac{1}{l'}\left\lvert A(l') - \frac{l'}{2} +\frac{1}{2\pi i }\sum_{0<|n|\leq k^{l'}} \zeta\left(\frac{2n\pi i}{\log k}\right) \frac{e^{2n\pi i p/q}}{n} \right\rvert \ll \gamma.
\]
By choosing $\gamma\to 0$, we conclude Theorem~\ref{theorem:generalization}. 
\end{proof}

\begin{acknowledgment}
The first author was financially supported by JST SPRING, Grant Number JPMJSP2125. The second author was financially supported by JSPS KAKENHI Grant Number JP22J00025 and JP22KJ0375. 
\end{acknowledgment}


\begin{thebibliography}{99}
    \bibitem{Borel} \'{E}. M. Borel, Les probabilit\'{e}s d\'{e}nombrables et leurs applications arithm\'{e}tiques, Rendiconti del Circolo Matematico di Palermo vol.27, 1909, 247--271.




   \bibitem{Good} A.~Good, Diskrete Mittel f\"{u}r einige Zetafunktionen, J. Reine Angew. Math. vol.303(304), 1978, 51--73.


    
    \bibitem{Ivic}  A.~Ivi\'{c}, The Riemann zeta-function. Theory and applications, Dover Publications, Inc., Mineola, NY, 2003. xxii+517 pp. 
    
    \bibitem{Kobayashi} H.~Kobayashi, A mean-square value of the Riemann zeta function over an arithmetic progression, arXiv:2212.06520.


    \bibitem{Montgomery-Vaughan} H. L. Montgomery and R. C. Vaughan, Multiplicative Number Theory I. Classical Theory, Cambridge University Press 2007.

    \bibitem{Ozbek-Steuding_2017} S.~S.~\"{O}zbek and J.~Steuding, The values of the Riemann zeta-function on arithmetic progressions, Analytic and Probabilistic Methods in Number Theory, Vilnius Univ. Leidykla, Vilnius, 2017, 149--164.
    
    \bibitem{Ozbek-Steuding_2019} S.~S.~\"{O}zbek and J.~Steuding, The values of the Riemann zeta-function on generalized arithmetic progressions, Arch. Math. (Basel) vol.112 (1), 2019, 53--59.

    \bibitem{Ridout} D.~Ridout, Rational approximations to algebraic numbers, Mathematika, vol.4, 1957, 125--131 

    \bibitem{Steuding-Wegert} J.~Steuding and E.~Wegert, The Riemann zeta function on arithmetic progressions, Exp. Math. vol.21 (3), 2012, 235--240.
    
    \bibitem{Titchmarsh} E. C. Titchmarsh, The Theory of the Riemann Zeta-function, Oxford Science Publication 1986.
\end{thebibliography}
\end{document}